\documentclass[11pt]{amsart}
\usepackage{latexsym,amssymb,amsmath}
\input epsf
\def\boxit#1{\vbox{\hrule height1pt\hbox{\vrule width1pt\kern3pt
  \vbox{\kern3pt#1\kern3pt}\kern3pt\vrule width1pt}\hrule height1pt}}

\def\trank{\text{rank}}

\def\BC{\mathbb C}

\def\BP{\mathbb P}
\def\pp#1{\mathbb P^{#1}}

\def\pp#1{{\mathbb P}^{#1}}
\def\tdim{\rm dim}
\def\hd{,...,}

\def\upperp{{}^\perp}

\def\be{\begin{equation}}
\def\ene{\end{equation}}
\def\aaa{{\bold a}}\def\bbb{{\bold b}}

\def\cS{{\mathcal S}}

\def\11{\mathbf 1}

\def\FS{{\mathfrak S}}

\def\a{\alpha}

\def\s{\sigma}

\def\ot{{\mathord{\,\otimes }\,}}
\def\op{{\mathord{\,\oplus }\,}}
\def\otc{{\mathord{\otimes\cdots\otimes}\;}}

\def\ra{{\mathord{\;\rightarrow\;}}}

\def\tdim{{\rm dim}\;}

\newtheorem{theorem}{Theorem}
\newtheorem{proposition}[theorem]{Proposition}
\newtheorem{lemma}[theorem]{Lemma}

\theoremstyle{definition}
\newtheorem{definition}[theorem]{Definition}

\theoremstyle{remark}
\newtheorem{remark}[theorem]{Remark}

\def\t{\tau}
 \def\aaa{{\bold a}}
\def\bbb{{\bold b}}
\def\ccc{{\bold c}}
\def\tmin{\operatorname{min}}

\begin{document}
\title{Kruskal's theorem}
\author{J.M. Landsberg}
%\date{August 2008}
\begin{abstract}   This is just a short proof of Kruskal's theorem regarding uniqueness of
expressions for tensors, phrased
in geometric language.
\end{abstract}
 \thanks{Supported by NSF grant DMS-DMS-0805782}
\email{jml@math.tamu.edu}
\maketitle

Let $A,B,C$ be complex vector spaces of dimensions $\aaa,\bbb,\ccc$.
Consider a tensor  $T\in A\ot B\ot C$ and say we have an expression
\be\label{kruskala}T=u_1\ot v_1\ot w_1+\cdots + u_r\ot v_r\ot w_r
\ene 
where $u_j\in A, v_j\in B, w_j\in C$,
and we want to know if the expression
is     unique up to re-ordering the factors (call this {\it essentially unique}). The {\it rank} of $T$ is by definition
the smallest such $r$ such that $T$ admits an expression of the form \eqref{kruskala}. For the tensor product of two vector spaces, an expression
as a sum of $r$ elements is never unique unless $r=1$. Thus an obvious necessary condition
for uniqueness is that we cannot be reduced to a two factor situation. For example,
an expression of the form
$$
T=a_1\ot b_1\ot c_1 + a_1\ot b_2\ot c_2 + a_3\ot b_3\ot c_3+\hdots +a_r\ot b_r\ot c_r
$$
where each of the sets $\{ a_i\},\{ b_j\}, \{ c_k\}$ are linearly independent
is not unique because of the first two terms. In other words if we consider for \eqref{kruskala}
the sets
$\cS_A=\{ [u_i]\}\subset \BP A$, $\cS_B=\{ [v_i]\}\subset\BP B$, $\cS_C=\{ [w_i]\}\subset \BP C$
each of the sets must consist of $r$ distinct points. 

\smallskip
 
We recall the classical fact:

\begin{proposition}\label{zerowash} Let $n>2$. 
Let $T\in A_1\otc A_n $ have rank $r$. Say
$T\in A_1'\otc A_n'$, where $A'_j\subseteq A_j$,   with at least one inclusion proper.
Then any expression   $T=\sum_{i=1}^{\rho}u^1_i\otc u^n_i$  with
some $u^s_j\not\in A_s'$ has $\rho>r$.
\end{proposition}

\begin{proof} Choose complements $A_t''$ so $A_t=A_t'\op A_t''$.
Assume $\rho=r$ and write $u^t_j={u^t_j}'+{u^t_j}''$ with ${u^t_j}'\in A_t'$, ${u^t_j}''\in A_t''$.
Then  
$T=\sum_{i=1}^{\rho}{u^1_i}'\otc {u^n_i}'$ so all the other terms must cancel.
Assume $\rho=r$, and say, e.g., some ${u^1_{j_0}}''\neq 0$.
Then
$\sum_{j=1}^r {u^1_j}''\ot ({u^2_j}'\otc {u^n_j}')=0$, but all the terms
$({u^2_j}'\otc {u^n_j}')$ must be linearly independent
in $A_2'\otc A_n'$ otherwise $r$ would not be minimal, thus all
the ${u^1_j}''$ must all be zero, a contradiction.
\end{proof}

\begin{definition} 
Let $\cS=\{x_1\hd x_p\}\subset \BP W$ be a set of points.
We say the points of $\cS$ are in $2$-general linear position
if no two points coincide, they are in $3$-general linear position
if no three lie on a line and more generally they are
in {\it $r$-general linear position} \index{general linear position}
if no $r-1$ of them lie in a $\pp{r-2}$.
We let the {\it Kruskal rank}\index{Kruskal rank} of $\cS$, 
$k_{\cS}$,  be the maximum number $r$ such that the points
of $\cS$ are in $r$-general linear position.
\end{definition}

If one chooses a basis for $W$ so that the points of $\cS$ can
be written as columns of a matrix (well defined up to rescaling 
columns), then $k_{\cS}$ will be the maximum number $r$ such that
all subsets of $r$ column vectors of the corresponding matrix
are linearly independent. (This was Kruskal's original definition.)

\begin{theorem}[Kruskal,\cite{MR0444690}]
Let $T\in A\ot B\ot C$. Say $T$ admits an expression
$T=\sum_{i=1}^ru_i\ot v_i\ot w_i$. Let
$\cS_A=\{ [u_i]\}$,$\cS_B=\{ [v_i]\}$,$\cS_C=\{ [w_i]\}$.
If
\be\label{kruskalbnd}
r\leq \frac 12(k_{\cS_A}+ k_{\cS_B}+k_{\cS_C})-1
\ene
then $T$ has rank $r$ and its expression as a rank $r$ tensor is
essentially unique.
\end{theorem}

Above, we saw a necessary condition for uniqueness is
that $k_{\cS_A}, k_{\cS_B},k_{\cS_C}\geq 2$ and it is an easy exercise to show  that if \eqref{kruskalbnd} holds, then
$k_{\cS_A}, k_{\cS_B},k_{\cS_C}\geq 2$. (Hint: {\it a priori} $k_{\cS_A}\leq r$.) 

Note that if $\aaa=\bbb=\ccc$ and $T: (A\ot B)^*\ra C$ and similar permutations
are surjective,  then it is very easy to
see such an expression is unique when $r=\aaa$. Kruskal's Theorem
extends the uniqueness to
  $\aaa\leq r\leq \frac 32\aaa-1$.

The key to the proof of Kruskal's theorem is the following lemma:

\begin{lemma}[Permutation lemma]
Let $W$ be a complex vector space and let $\cS=\{ p_1\hd p_r\}$, $\tilde \cS=\{ q_1\hd q_r\}$ be
sets of points in $\BP W$ and assume no two points
of $\cS$ coincide (i.e., that $k_{\cS}\geq 2$) and that $\langle \tilde \cS\rangle=W$.
If   all hyperplanes $H\subset \BP W$ that have the property
that they contain at least  $\tdim ( H)+1$ points of $\tilde \cS$ also have
the property that
$\# ( \cS\cap H)\geq \# ( \tilde \cS\cap H)$, then $\cS=\tilde \cS$.
\end{lemma}

If one chooses a basis for $W=\BC^n$ and writes the two sets of points
as matrices $M,\tilde M$,  then the hypothesis can be rephrased (in fact this was
the original phrasing) as to say that for all $x\in \BC^n$ such that the
number of nonzero elements of the vector
${}^t\tilde Mx$ is less than $r-{\rm{rank}}(\tilde M)+1$ also has
the property that 
 the
number of nonzero elements of the vector
${}^t\tilde Mx$ is at most
 the
number of nonzero elements of the vector
${}^t  Mx$.  To see the correspondence, the vector $x$ should be
thought of as point of $W^*$ giving an equation of $H$, zero elements
of the vector ${}^t\tilde Mx$ correspond to columns that pair with $x$
to be zero, i.e., that satisfy an equation of $H$, i.e., points that are contained
in $H$.

Note a slight discrepancy with the original formulation: we have assumed
$\langle \tilde \cS\rangle =W$ so ${\rm{rank}}(\tilde M)=n$. Our hypothesis is
slightly different, but it is all that is needed by Proposition \ref{zerowash}. Had we not
assumed this, there would be trivial cases to eliminate at each step of our
proof.

\begin{proof}
First note that if one replaces \lq\lq hyperplane\rq\rq\  by \lq\lq point\rq\rq\ in the hypotheses
of the lemma, then it  follows immediately as the points of $\cS$ are distinct.  
The proof will proceed by induction going from hyperplanes to points.
Assume   $(k+1)$-planes $M$ that have the property that they contain at least $k+2$ points of $\tilde \cS$
also have the property that $\# ( \cS\cap M)\geq \# ( \tilde \cS\cap M)$ and we will show the
same holds for    $k$-planes.
Fix a $k$-plane $L$ containing $\mu\geq k+1$ points of $\tilde \cS$, and let
$\{ M_{\a}\}$ denote the set of $k+1$ planes containing $L$ and at least $\mu+1$
elements of $\tilde \cS$.
We have
\begin{align*}
&\#(\tilde\cS\cap L)+\sum_{\a}\#(\tilde \cS\cap (M_{\a}\backslash L))=R\\
&\#( \cS\cap L)+\sum_{\a}\#(  \cS\cap (M_{\a}\backslash L))\leq R
\end{align*}
the first line because every point of $\tilde\cS$ not in $L$ is in exactly
one  $M_{\a}$ and the second because every point of $ \cS$ not in $L$ is in at most
one  $M_{\a}$.
Rewrite these as
\begin{align*}
&(\# M_{\a}-1)\#(\tilde\cS\cap L)-\sum_{\a}\#(\tilde \cS\cap M_{\a} )=-R\\
&(\# M_{\a}-1)\#( \cS\cap L)-\sum_{\a}\#(  \cS\cap  M_{\a} )\geq -R
\end{align*}
But by our induction hypothesis $\sum_{\a}\#(  \cS\cap  M_{\a} )\geq 
\#(\tilde \cS\cap M_{\a} )$ so putting
the two lines  together, we obtain the   result for $L$.
\end{proof}

\begin{proof}[Proof of theorem]
Given decompositions $\phi=\sum_{j=1}^ru_j\ot v_j\ot w_j,\tilde\phi=\sum_{j=1}^r\tilde u_j\ot \tilde v_j\ot \tilde w_j$ of length $r$ we want to show they
are essentially the same. (Note that if there were a decomposition $\tilde\phi$ of
length e.g., $r-1$, we could construct from it a decomposition of length $r$
by replacing $\tilde u_1\ot \tilde v_1\ot \tilde w_1$ by
$\frac 12\tilde u_1\ot \tilde v_1\ot \tilde w_1+\frac 12\tilde u_1\ot \tilde v_1\ot \tilde w_1$, so
uniqueness of the length $r$ decomposition implies the rank is $r$.)
We first show $\cS_A=\tilde\cS_A$,$\cS_B=\tilde\cS_B$,$\cS_C=\tilde\cS_C$.
By symmetry it  is sufficient to prove the last statement. 
By the permutation lemma it is sufficient to show
that    if $H\subset \BP C$ is a hyperplane
such that
$\#(\tilde\cS_C\cap H)\geq \ccc -1$ then
$\#( \cS_C\cap H)\geq \#(\tilde\cS_C\cap H)$ because  
we already know $k_{\cS_C}\geq 2$.

Recall the classical fact about matrices (due to Sylvester): if $M\in A\ot B$ and
$U\subset A$, $V\subset B$, then
$$
\trank(M)\geq \trank(M|_{U\upperp\times B^*}) +
\trank(M|_{A^*\times V\upperp} )  -\trank(M|_{U\upperp\times V\upperp}).
$$

Let $A_{H}:=\langle u_j\mid [w_j]\not\in H\rangle$,  
$B_{H}:=\langle v_j\mid [w_j]\not\in H\rangle$
\begin{align*}
\#(\tilde \cS_c\not\subset H)&\geq
\trank (T(H\upperp))
\\
&\geq \trank(T(H\upperp)|_{A_H\upperp\times B^*}) + 
  \trank(T(H\upperp)|_{A^*\times B_H\upperp})-\trank(T(H\upperp)|_{A_H\upperp\times B_H\upperp})
\\
&\geq   \tmin(k_A, \#( \cS_C\not\subset H)) + \tmin(k_B, \#( \cS_C\not\subset H)) -
 \#( \cS_C\not\subset H)
\end{align*}
where the last line follows by the definition of Kruskal rank.
Finally we need to show that 
$\#( \cS_C\not\subset H)\leq \tmin(k_A,k_B)$. But this follows because  
$$
r-\#( \cS_C\not\subset H)=\#( \cS_C \subset H) \geq \ccc-1\geq k_C-1\geq 2r-k_A-k_B+1
$$
i.e., $k_A+k_B-\#(\cS_C\not\subset H)\geq r+1$,
which  can only hold if 
$\#( \cS_C\not\subset H)\leq \tmin(k_A,k_B)$.
 
\medskip

Now that we have $\cS_A=\tilde \cS_A$ etc.. , say we have two expressions
\begin{align*}
T&=u_1\ot v_1\ot w_1+\cdots + u_r\ot v_r\ot w_r\\
T&=u_1\ot v_{\s(1)}\ot w_{\t(1)}+\cdots + u_r\ot v_{\s(r)}\ot w_{\t(r)}
\end{align*}
for some $\s,\t\in \FS_r$.
First observe that if $\s=\t$ then we are reduced to the two factor case which is easy, i.e., 
if $T\in A\ot B$ of rank $r$ has expressions $T=a_1\ot b_1+\cdots + a_r\ot b_r$ and
$T=a_1\ot b_{\s(1)}+\cdots + a_r\ot b_{\s(r)}$, then it is easy to see that $\s=Id$. 

So assume $\s\neq \t$, then there exists a smallest $j_0\in \{ 1\hd r\}$ such that
$\s(j_0)=:s_0\neq t_0:=\t(j_0)$.
We   claim there exist subsets $S,T\subset \{1\hd r\}$ with the properties
\begin{itemize}
\item
$s_0\in S$, $t_0\in T$, 
\item $S\cap T=\emptyset$, 
\item $\#(S)\leq r-k_{\cS_B}+1$, $\#(T)\leq r-k_{\cS_C}+1$  and
\item $\langle v_j\mid j\in S^c\rangle=: H_S\subset B$,
$\langle w_j\mid j\in T^c\rangle=: H_T\subset C$ are hyperplanes.
\end{itemize}
Here $S^c=\{1\hd r\}\backslash S$.

To prove the claim take a hyperplane $H_T\subset C$ containing $w_{s_0}$ but not
containing $w_{t_0}$, and let $T^v$ be the set of indices of the $w_j$ contained in $H_T$, so
in particular $\#(T^c)\geq k_{\cS_C}-1$ insuring the cardinality bound for $T$.
Now consider the linear space $\langle v_t\mid t\in T\rangle \subset B$. 
Since $\#(T)\leq r-k_{\cS_C}+1\leq k_{\cS_B}-1$ (the last inequality because $k_{\cS_A}\leq r$),
adding any vector of $\cS_B$ to $\langle v_t\mid t\in T\rangle$ would increase
its dimension, in particular, $v_{s_0}\notin 
\langle v_t\mid t\in T\rangle$.
Thus there exists a hyperplane $H_S\subset B$  containing $\langle v_t\mid t\in T\rangle$
and not containing $v_{s_0}$. Let $S$
be the set of indices of the $v_j$ contained in $H_S$.
Then $S,T$ have the desired properties. 

Now by construction $T|_{H_S\upperp \times H_T\upperp}=0$, which implies there is
a nontrivial linear relation among the $u_j$ for the $j$ appearing in $S \cap T $, but
this number is at most
$\tmin (r-k_{\cS_B}+1,r-k_{\cS_C}+1)$ which is   less than  $k_{\cS_A}$.  
\end{proof}

\begin{remark} There were several inequalities used in the proof that were far from sharp. In fact,
Kruskal proves versions of his theorem with weaker hypotheses designed to be more efficient
regarding the use of the inequalities.
\end{remark}

\begin{remark} The proof above is essentially Kruskal's. The reduction from a 16 page proof to
the 2 page proof above is mostly due to writing statements invariantly rather than in coordinates.
\end{remark}

More generally, Kruskal shows that for $d$ factors, if $\sum_{i=1}^d\cS_{k_i}\geq 2r+d-1$ then   uniqueness holds.

\subsection*{Acknowledgments}This short note is an outgrowth of the AIM workshop 
{\it Geometry and representation theory of tensors for computer science, statistics and other areas}   July 21-25, 2008,
and the author  gratefully thanks AIM and the other participants of the workshop, in particular
L. De Lathauwer and P. Comon who encouraged the writeup. It will appear   in the forthcoming book {\it Geometry of Tensors: Applications to complexity, statistics
and engineering} with J. Morton.

\bibliographystyle{amsplain}
%\nocite{*}
\bibliography{Lmatrix}

\def\cdprime{$''$} \def\Dbar{\leavevmode\lower.6ex\hbox to 0pt{\hskip-.23ex
  \accent"16\hss}D} \def\cprime{$'$} \def\cprime{$'$} \def\cprime{$'$}
  \def\cprime{$'$} \def\Dbar{\leavevmode\lower.6ex\hbox to 0pt{\hskip-.23ex
  \accent"16\hss}D} \def\cprime{$'$}
\providecommand{\bysame}{\leavevmode\hbox to3em{\hrulefill}\thinspace}
\providecommand{\MR}{\relax\ifhmode\unskip\space\fi MR }
% \MRhref is called by the amsart/book/proc definition of \MR.
\providecommand{\MRhref}[2]{%
  \href{http://www.ams.org/mathscinet-getitem?mr=#1}{#2}
}
\providecommand{\href}[2]{#2}
\begin{thebibliography}{1}

\bibitem{MR0444690}
Joseph~B. Kruskal, \emph{Three-way arrays: rank and uniqueness of trilinear
  decompositions, with application to arithmetic complexity and statistics},
  Linear Algebra and Appl. \textbf{18} (1977), no.~2, 95--138. \MR{MR0444690
  (56 \#3040)}

\end{thebibliography}
\end{document}